\documentclass[leqno, 12pt]{article}
\usepackage{amsthm}
\usepackage{amssymb}
\usepackage{amsmath}
\usepackage[english]{babel}
\usepackage[utf8]{inputenc}
\usepackage[T1]{fontenc}
\usepackage[normalem]{ulem}
\usepackage{graphicx}
\usepackage{dsfont}
\usepackage{enumitem}
\usepackage{calc}
\usepackage{caption}
\usepackage{subcaption}
\usepackage{hyperref}
\usepackage{times}
\usepackage{float}
\usepackage{marvosym}
\usepackage{hyperref}
\usepackage{xhfill}
\usepackage{tikz}
\usetikzlibrary{shapes,fit,positioning,calc,matrix,arrows}
\tikzset{
  optree/.style={scale=.5,thick,grow'=up,level distance=10mm,inner sep=1pt},
  comp/.style={draw=none,circle,fill,line width=0,inner sep=0pt},
  dot/.style={draw,circle,fill,inner sep=0pt,minimum width=3pt},
  circ/.style={draw,circle,inner sep=1pt,minimum width=4mm},
  emptycirc/.style={draw,circle,inner sep=1pt,minimum width=2mm},
  root/.style={level distance=10mm,inner sep=1pt},
  leaf/.style={draw=none,circle,fill,line width=0,inner sep=0pt},
  nodot/.style={draw,circle,inner sep=1pt},
}
\usepackage{tikz-cd}
\usepackage[top=0.80in, bottom=1in, left=1in, right=1in]{geometry}

\newtheorem*{cnj*}{Conjecture}
\newtheorem*{thm*}{Theorem}
\newtheorem{thm}{Theorem}[section]
\newtheorem{crl}[thm]{Corollary}
\newtheorem{prp}[thm]{Proposition}

\theoremstyle{definition}
\newtheorem{dfn}[thm]{Definition}

\theoremstyle{remark}
\newtheorem{rmk}[thm]{Remark}

\newcommand{\PL}{\mathrm{PreLie}}
\newcommand{\Lie}{\mathrm{Lie}}
\newcommand{\CLie}{\mathrm{CycLie}}
\newcommand{\Com}{\mathrm{Com}}
\newcommand{\DLie}{\mathrm{Lie_2}}
\newcommand{\DCom}{\mathrm{Com_2}}

\newcommand{\PostLie}{\mathrm{PostLie}}
\newcommand{\Perm}{\mathrm{Perm}}
\newcommand{\Mag}{\mathrm{Mag}}
\newcommand{\CMag}{\mathrm{ComMag}}

\newcommand{\ANil}{\mathrm{SkewNil_2}}

\newcommand{\NAC}{\mathrm{NAC_2}}
\newcommand{\Dup}{\mathrm{Dup}}
\newcommand{\TriDend}{\mathrm{TriDend}}
\newcommand{\TriAss}{\mathrm{TriAss}}
\newcommand{\ComTriAss}{\mathrm{ComTriAss}}

\newcommand{\ash}{\textnormal{¡}}

\newcommand{\bbone}{\mathds{1}}
\newcommand{\Cob}{\mathbf{\Omega}}
\newcommand{\BB}{\mathbf{B}}
\newcommand{\rCob}{\bar{\mathbf{\Omega}}}
\newcommand{\rBar}{\bar{\mathbf{B}}}

\newcommand{\cC}{\mathcal{C}}
\newcommand{\cP}{\mathcal{P}}
\newcommand{\cQ}{\mathcal{Q}}
\newcommand{\cX}{\mathcal{X}}
\newcommand{\cY}{\mathcal{Y}}
\newcommand{\cT}{\mathcal{T}}
\newcommand{\cS}{\mathcal{S}}
\newcommand{\cR}{\mathcal{R}}
\newcommand{\cL}{\mathcal{L}}
\newcommand{\cM}{\mathcal{M}}

\newcommand{\unit}{\mathrm{e}}

\newcommand{\Fun}{\mathrm{Fun}}
\newcommand{\Hom}{\mathrm{Hom}}
\newcommand{\Set}{\mathrm{Set}}
\newcommand{\Vect}{\mathrm{Vect}}

\newcommand{\Spe}{\mathrm{Spe}}
\newcommand{\bbB}{\mathbb{B}}
\newcommand{\id}{\mathrm{id}}

\newcommand{\Sym}{\mathfrak{S}}

\title{Homological freeness criterion of Koszul operad and application to Cohen-Macalayness of posets}
\date{}
\author{Paul Laubie}

\begin{document}

\maketitle

\abstract{We show a variation of the usual homological freeness criterion for operadic modules over a Koszul operad. 
We then apply this result to decorated partition posets for some operads, showing that their augmentation is Cohen-Macaulay and computing its homology. 
This work answers several open questions asked by Bérénice Delcroix-Oger and Clément Dupont in a recent article.}

\section{Introduction}

The homological freeness criterion for graded modules was introduced by Eilenberg in the seminal paper \cite{HDim}.
This criterion is known to hold for operadic module, indeed a graded operadic right $\cP$-module $\cQ$ is free if and only if $\BB(\cQ,\bar{\cP},\bbone)$ is acyclic.
For example, this criterion is used in \cite{FunPBW} by Dotsenko, and Tamaroff to show the freeness of some operadic right modules.
We show a variation of this result for Koszul operads.
\begin{thm*}\ref{crl:main}
	Let $\cP\to \cQ$ a morphism of graded operads with $\cP$ and $\cQ$ Koszul.
	This morphism induces a free left (resp. right) $\cP$-module structure on $\cQ$ if and only if the homology of $\Cob(\cP^\ash,\cQ^\ash,\bbone)$ (resp. $\Cob(\bbone,\cQ^\ash,\cP^\ash)$) is concentrated in the diagonal. 
	Moreover, in this case, $H(\Cob(\cP^\ash,\cQ^\ash,\bbone))$ (resp. $H(\Cob(\bbone,\cQ^\ash,\cP^\ash))$) computes the generators of $\cQ$ as a free left (resp. right) $\cP$-module.
\end{thm*}
One should notice that the roles of $\cP$, and $\cQ$ are reverse in the criterion we give, hopefully explaining the additional Koszul hypothesis.

We then apply this result to several $\Lie$-operads. Mainly to the operads $\PL$, $\DLie$, and $\PostLie$, that are known to be free over $\Lie$. This allows us to answer several questions raised in a recent article \cite{OpPOS} of Delcroix-Oger, and Dupont on operadic posets.
More precisely, we get the following results:
\begin{thm*}\ref{thm:pl}
	The poset homology of $\Pi^\Perm_+$ is isomorphic to the species $\cT(\CLie)$ (up to suspension).	
\end{thm*}
\begin{thm*}\ref{thm:dlie}
	The posets $ ^\DCom\Pi^+$, and $\Pi^\DCom_+$ are Cohen-Macaulay.
\end{thm*}
\begin{thm*}\ref{thm:postl}
	The poset $\Pi^\ComTriAss_+$ is Cohen-Macaulay. Moreover, its homology is isomorphic to the species $\Mag$ (up to suspension).
\end{thm*}

\textbf{Organisation of the article.}
We begin by some recollections on the theory of linear species, and algebraic operads. 
First, we recall the first definitions of linear species, and algebraic operads. 
Then we give a quick introduction to operadic homological algebra, with the twisting morphisms, and the simplicial (co)bar construction.

In the following section, we introduce a slight generalisation of the (co)bar complexes, allowing us to show our main theorem.

Finally, we apply this theorem to the several $\Lie$-operads, in particular to $\PL$, $\DLie$, and $\PostLie$.
This allows us to show the Cohen-Macaulayness of several posets, and computes the homology of two of them.

\section{Recollection on species and operads}

\subsection{Linear species}

All vector spaces and chain complexes in this article are defined over a ground field $\Bbbk$ of zero characteristic. 
We freely use terminology of species of structures in the context of algebraic operads, we refer the reader to \cite{CombSpe,TLSpe,MonFSpe,MonFSpe} for introduction to the theory of species, and of algebraic operads. 

Let $\bbB$ the groupoid of finite sets with the bijections, we recall that the category of \emph{combinatorial species} is defined as $\Spe=\Fun(\bbB,\Set)$, and the category of \emph{linear species} as $\Vect\Spe=\Fun(\bbB,\Vect)$, with $\Set$ the category of sets, and $\Vect$ of vector spaces. 
We recall that the category $\Vect\Spe$ admits a monoidal structure the \emph{plethysm} admitting the following formula, let $\cR,\cS$ two linear species such that $\cR(\emptyset)=\cS(\emptyset)=0$, then
$$\cR\circ\cS(N)=\bigoplus_{P\vdash N}\cR(P)\otimes\bigotimes_{q\in P}\cS(q)$$
which is a categorification of the Faà di Bruno formula. Here, $P\vdash N$ denotes that $P$ is a partition of $N$.
The unit of the plethysm is the linear species $\bbone$ defined by $\bbone(\{\bullet\})=\Bbbk$, and $\bbone(N)=0$ if $|N|\neq 1$.
An \emph{algebraic operad} is a monoidal object in $\Vect\Spe$ for this monoidal structure.

For $\cR$, and $\cS$ two linear species such that $\cR(\emptyset)=\cS(\emptyset)=0$, the \emph{infinitesimal composition} denoted $\cR\circ'\cS$ is defined by:
$$\cR\circ'\cS(N)=\bigoplus_{P\vdash N}\cR(P)\otimes\bigoplus_{q\in P}\cS(q).$$
We should notice that a morphism $\bbone\to\cS$ induces a morphism $\cR\circ'\cS\to\cR\circ\cS$. Moreover, $\cS\circ'\cS$ is the degree $2$ part of $\cT(\cS)$, the free operad over $\cS$.

One can also define \emph{weight graded linear species}, and \emph{differential graded linear species} which behave as expected.
In particular, the homology of a differential graded linear species can be computed pointwise, $H(\cS)(N)=H(\cS(N))$.
In the sequel, we will mainly consider linear species that are both differential graded, and weight graded such that the differential respects the weight, such species will be called \emph{weighted dg species}. 
To compute the homology of such a species, we will use the tool of spectral sequences, and we refer to \cite{SpSeq} for an introduction to spectral sequences.
We point out that since the homology is computed pointwise, the weight grading being pointwise bounded is enough for the associated spectral sequence to converge.

\subsection{Operadic twisting morphisms}

Let us give a very quick introduction to operadic twisting morphisms, we refer to \cite[Section 6.4]{AlgOp} for more details.
We will use the terminology dg for differential graded.

Let $\cC$ a dg cooperad, $\cP$ a dg operad, 

\begin{dfn}
	The \emph{operadic convolution algebra} $\Hom(\cC,\cP)$ is the vector space of species morphisms from $\cC$ to $\cP$, endowed with the product $\star$ given by:
	$$f\star g: \cC\to\cC\circ'\cC\overset{f\circ' g}{\longrightarrow}\cP\circ'\cP\to\cP.$$
	This is a dg pre-Lie algebra, where the differential is the usual differential $\partial$ on morphisms of chain complexes.
\end{dfn}

\begin{dfn}
	A map $\alpha:\cC\to\cP$ is a \emph{twisting morphism} if
	$$\partial\alpha + \alpha\star\alpha=0.$$
\end{dfn}

Let $\alpha:\cC\to\cP$ a degree $-1$ morphism of species.
Let $d^l_\alpha$ the unique derivation which extend the map below as a derivation on $\cC\circ\cP$.
$$\cC\to\cC\circ'\cC\to\cC\circ'\cP\to\cC\circ\cP$$
Let $d_\alpha=d_{\cC\circ\cP}+d^l_\alpha$.

Similarly, we can define $d^r_\alpha$ on $\cP\circ\cC$ extending:
$$\cC\to\cC\circ'\cC\to\cP\circ'\cC\to\cP\circ\cC$$
and $d_\alpha=d_{\cP\circ\cC}+d^r_\alpha$.

\begin{prp}[\cite{AlgOp} Lemma 6.4.12]
	The following propositions are equivalent:
	\begin{itemize}
		\item $\alpha$ is a twisting morphism,
		\item $d_\alpha^2=0$ on $\cC\circ\cP$, and
		\item $d_\alpha^2=0$ on $\cP\circ\cC$.
	\end{itemize}
\end{prp}

\subsection{(Co)Bar complexes}\label{sec:Kp}

We refer to \cite{KDPartHom,CompBar,AlgOp} for an introduction to operadic (co)bar constructions and comparison of the different constructions. In this article, we will use the simplicial operadic (co)bar construction.

Let $\cP$ an operad, and $\cL$, $\cR$ respectively a left and a right $\cP$-module.
We recall that it provides us the following maps: 
\begin{itemize}
	\item the operadic product $\gamma:\cP\circ\cP\to\cP$,
	\item the unit $\unit:\bbone\to\cP$,
	\item the left action $\rho_\cL:\cL\circ\cP\to\cP$, and
	\item the left action $\rho_\cR:\cP\circ\cR\to\cP$.
\end{itemize}
Thus, it allows us to construct the following simplicial species denoted $\BB(\cR,\cP,\cL)$ such that:
$$\BB_n(\cR,\cP,\cL)=\cR\circ\cP^{\circ n}\circ\cL$$
and the face and degeneracy maps a given by:
\begin{itemize}
	\item $d_{n,0}=(\rho_\cR,\id,\dots,\id)$,
	\item $d_{n,n}=(\id,\dots,\id,\rho_\cL)$,
	\item $d_{n,i}=(\id,\dots,\id,\gamma,\id,\dots,\id)$ with $\gamma$ in position $i+1$, and for $i\neq 0$ or $n$,
	\item $s_{n,j}=(\id,\dots,\id,\unit,\id,\dots,\id)$ with $\unit$ in position $j+1$.
\end{itemize}

We identify $\BB(\cR,\cP,\cL)$ with its associated differential graded linear species.
We denote by $\rBar(\cP)=\BB(\bbone,\cP,\bbone)$ the reduced bar construction of $\cP$. Furthermore, we point out that this is canonically a cooperad, moreover, we have a canonical twisting morphism $\kappa:\cP\to\rBar(\cP)$. 

For $\cC$ a cooperad, and $\cL$, $\cR$ a left and a right $\cC$-comodule, we get the cobar construction $\Cob(\cR,\cC,\cL)$. 
We denote by $\rCob(\cC)=\Cob(\bbone,\cC,\bbone)$ the reduced cobar construction of $\cC$, which is canonically an operad, and admits a canonical twisting morphism $\kappa:\cC\to\rCob(\cC))$.

Let us recall the following classical result:

\begin{thm}[\cite{AlgOp} Lemma 6.5.14]
	Let $\cP$ an operad. Then 
	\begin{itemize}
		\item $\BB(\cP,\cP,\bbone)\simeq(\cP\circ\rBar(\cP),d_\kappa)\simeq\bbone$, and
		\item $\BB(\bbone,\cP,\cP)\simeq(\rBar(\cP)\circ\cP,d_\kappa)\simeq\bbone$.
	\end{itemize}
	The same results hold for a cooperad.
\end{thm}

\begin{thm}[\cite{AlgOp} Theorem 6.6.5]
	Let $\cP$ an operad, then $\cP\simeq\rCob(\rBar(\cP))$.
	The same result holds for a cooperad.
\end{thm}

In the application section of this article, some Koszul operads are studied. 
We recall that Koszul operads are quadratic, in particular, they are graded.
However, let us skip the definition of a Koszul operad since we will only use the following property of Koszul operads. 

\begin{thm}[\cite{AlgOp} Proposition 7.3.2]
	Let $\cP$ a Koszul operad concentrated in homological degree $0$, then $\rBar(\cP)$ is diagonal.
	Dually, if $\cC$ is a Koszul cooperad concentrated in homological degree $0$, then $\rCob(\cC)$ is diagonal.
\end{thm}

Through this article, the reader unfamiliar with operadic Koszul theory can consider the above result as the definition of being Koszul for an operad (resp. cooperad) concentrated in homological degree $0$.

\section{Main result}

\subsection{The chained (co)bar constructions}

Let us introduce a slight generalisation of the (co)bar construction that we will use in the proof of the theorems.

Let:
\begin{itemize}
	\item $\cP$, and $\cQ$ two operads,
	\item $\cL$ a left $\cQ$-module,
	\item $\cR$ a right $\cP$-module, and
	\item $\cM$ a $\cP$-$\cQ$-bimodule, meaning that $\cM$ is a left $\cP$-module and a right $\cQ$-module such that the following diagram commutes:
		\[\begin{tikzcd}
			{\cP\circ\cM\circ\cQ} & {\cP\circ\cM} \\
			{\cM\circ\cQ} & {\cM}
			\arrow[from=1-1, to=1-2]
			\arrow[from=1-1, to=2-1]
			\arrow[from=1-2, to=2-2]
			\arrow[from=2-1, to=2-2]
		\end{tikzcd}\]
\end{itemize}

Let us defined the \emph{chained bar construction} $\BB(\cR,\cP,\cM,\cQ,\cL)$ as the bisimplicial species given by:
$$\BB_{i,j}(\cR,\cP,\cM,\cQ,\cL)=\cR\circ\cP^{\circ i}\circ\cM\circ\cQ^{\circ j}\circ\cL$$
The horizontal face maps are given by the operadic product of $\cP$, and the $\cP$-module maps of $\cR$ and $\cM$. 
The horizontal degeneracy maps are given by the unit of $\cP$.
The vertical face maps are given by the operadic product of $\cQ$, and the $\cQ$-module maps of $\cM$ and $\cL$.
Finally, the vertical degeneracy maps are given by the unit of $\cQ$.
One can easily check that this is indeed a bisimplicial species.

Same as with the bar construction, we identify $\BB(\cR,\cP,\cM,\cQ,\cL)$ with its associated double dg species. 
For:
\begin{itemize}
	\item $\cC$, and $\cC'$ two operads,
	\item $\cL$ a left $\cC'$-comodule,
	\item $\cR$ a right $\cC$-comodule, and
	\item $\cM$ a $\cC$-$\cC'$-bicomodule, meaning that $\cM$ is a left $\cC$-comodule and a right $\cC'$-module such that the comodule maps commutes,
\end{itemize}
we have the dual construction, the \emph{chained cobar construction}, $\Cob(\cR,\cC,\cM,\cC',\cL)$ which is a bicosimplicial species. We identify it with its associated double dg species.

\begin{prp}~
	\begin{itemize}
		\item The total chain complex $\BB(\cR,\cP,\cM,\cQ,\cL)$ is isomorphic to $\BB(\cR,\cP,\bbone)\circ\BB(\cM,\cQ,\cL)$ as graded linear species. 
			Moreover, if the $\cP$-module map $\cP\circ\cM\to\cM$ is trivial, meaning that it is the identity on $\bbone$, and zero on $\bar{\cP}$ under the decomposition $\cP=\bbone\oplus\bar{\cP}$, then we get an isomorphism of dg species.
		\item The total chain complex $\BB(\cR,\cP,\cM,\cQ,\cL)$ is isomorphic to $\BB(\cR,\cP,\cM)\circ\BB(\bbone,\cQ,\cL)$ as graded linear species. 
			Moreover, if the $\cQ$-module map $\cM\circ\cQ\to\cM$ is trivial, then we get an isomorphism of dg species.
	\end{itemize}
	Dually, the same results hold for the chained cobar construction.
\end{prp}

\begin{proof}
	By definition, we have 
	$$\BB_n(\cR,\cP,\cM,\cQ,\cL)=\bigoplus_{i+j=n}\cR\circ\cP^{\circ i}\circ\cM\cQ^{\circ j}\circ\cL$$
	which is exactly the degree $n$ component of both $\BB(\cR,\cP,\bbone)\circ\BB(\cM,\cQ,\cL)$, and $\BB(\cR,\cP,\cM)\circ\BB(\bbone,\cQ,\cL)$.

	To get a morphism of dg species, one should check that the differentials are the same. A quick inspection show that this is the case if and only if the according module map is trivial.
\end{proof}

One can naturally extend this construction to any sequence of operads (resp. cooperads) $(\cP_1,\dots,\cP_n)$ chained together via a sequence of operadic bimodules (resp. bicomodules) $(\cM_{1,2},\dots,\cM_{i,i+1},\dots,\cM_{n-1,n})$ such that $\cM_{i,i+1}$ is a $\cP_{i}$-$\cP_{i+1}$-module (resp. comodule). 

\subsection{Freeness theorem}

Through this section, let $\cC$, and $\cC'$ two weighted cooperads endowed with a morphism of weighted cooperads $\varphi:\cC'\to\cC$.
We set $\cQ=H(\Cob(\cC))$ and $\cP=H(\Cob(\cC'))$, those are homologically weighted dg operads endowed with a morphism of weighted dg operads $\psi:\cP\to\cQ$.

\begin{prp}
	Let $\pi:\cC\to\cQ$, and $\kappa:\cC'\to\cP$ the canonical twisting morphisms. Then the following diagram commutes:
	\[\begin{tikzcd}
		{\cC'} & {\cC}\\
		{\cP} & {\cQ}
		\arrow["\varphi", from=1-1, to=1-2]
		\arrow["\kappa", from=1-1, to=2-1]
		\arrow["\pi", from=1-2, to=2-2]
		\arrow["\psi", from=2-1, to=2-2]
	\end{tikzcd}\]
	Moreover, the diagonal map $\alpha$ is a twisting morphism.
\end{prp}

\begin{proof}
	The commutativity of this diagram is tautological by definition of $\psi$.
	We have $\pi\star\pi=0$ in the convolution algebra $\Hom(\cC,\cQ)$, and $\alpha=\varphi^*(\pi)$ with $\varphi^\ast:\Hom(\cC,\cQ)\to\Hom(\cC',\cQ)$.
	Thus, $\alpha\star\alpha=0$.
\end{proof}

\begin{thm}
	Let $\cX$, and $\cY$ two diagonal species.
	\begin{itemize}
		\item If $\cQ=\cP\circ\cX$ as a left $\cP$-module, then:
		$$H(\Cob(\cC',\cC,\bbone))=H(\cC'\circ\cQ,d_\alpha)=\cX$$
		\item If $\cQ=\cY\circ\cP$ as a right $\cP$-module, then:
		$$H(\Cob(\bbone,\cC,\cC'))=H(\cQ\circ\cC',d_\alpha)=\cY$$
	\end{itemize}
\end{thm}

\begin{proof}
	Let us use a spectral sequence argument. Let us endow $\Cob(\cC',\cC,\bbone)$ with the filtration induced by concentrating $\cC$ in weight $1$, and $\cC'$ and $\bbone$ in weight $0$. 
	The differential $d_\Cob$ decreases along this filtration and we can study the induced spectral sequence. By definition, we have that:
	$$E^1=(\cC'\circ\cQ,d_\alpha)$$
	Since $\cQ=\cP\circ \cX$, we have that:
	$$E^1=(\cC'\circ\cQ,d_\alpha)=(\cC'\circ\cP\circ\cX,d_\kappa)$$
	Because $H(\cC'\circ\cP,d_\kappa)=\bbone$, we have that $E^2=\cX$. 
	Since $\cX$ is diagonal, the spectral sequence collapses at page $2$.
	Moreover, the spectral sequence converges since the grading is pointwise bounded.\\
	The proof of the second claim is the same.
\end{proof}

\begin{thm}
	Assume that $\cP$ and $\cQ$ are diagonal.
	\begin{itemize}
		\item If $H(\Cob(\cC',\cC,\bbone))=\cX$ with $\cX$ diagonal, then $\cQ=\cP\circ\cX$ as left $\cP$-module. 
		\item If $H(\Cob(\bbone,\cC,\cC'))=\cY$ with $\cY$ diagonal, then $\cQ=\cY\circ\cP$ as right $\cP$-module.
	\end{itemize}
\end{thm}

\begin{proof}
	Let us consider the chain complex $\Cob(\bbone,\cC',\cC',\cC,\bbone)$, we recall that:
	$$\Cob_{i,j}(\bbone,\cC',\cC',\cC,\bbone)=\bbone\circ {\cC'}^{\circ i}\circ \cC'\circ \cC^{\circ j}\circ \bbone$$
	The first page of the spectral sequence induced by the $j$-grading is $H(\Cob(\bbone,\cC',\cC'))\circ H(\Cob(\bbone,\cC,\bbone))$ which is equal to $\cQ$. 
	Since $\cQ$ is diagonal, the spectral sequence collapses at the first page, and: 
	$$H(\Cob(\bbone,\cC',\cC',\cC,\bbone))=\cQ.$$
	Similarly, the first page of the spectral sequence induced by the $i$-grading is $H(\Cob(\bbone,\cC',\bbone))\circ H(\Cob(\cC',\cC,\bbone))$ which is equal to $\cP\circ\cX$. Since $\cP$, and $\cX$ are diagonal, it collapses at the first page, and 
	$$H(\Cob(\bbone,\cC',\cC',\cC,\bbone))=\cP\circ \cX.$$
	Thus:
	$$\cQ=\cP\circ\cX$$
	The proof of the second claim is the same.
\end{proof}

We recall that if $\cC'$ (resp. $\cC$) is Koszul, and concentrated in homological degree $0$, then $\cP$ (resp. $\cQ$) is diagonal. Thus, we get the following corollary: 

\begin{crl}\label{crl:main}
	Assume that $\cC'$ and $\cC$ are Koszul, and concentrated in homological degree $0$. 
	\begin{itemize}
		\item We have $\cQ$ is free as a left $\cP$-module if and only if $H(\Cob(\cC',\cC,\bbone))$ is diagonal. Moreover, in this case we have $\cQ=\cP\circ H(\Cob(\cC',\cC,\bbone))$.
		\item We have $\cQ$ is free as a right $\cP$-module if and only if $H(\Cob(\bbone,\cC,\cC'))$ is diagonal. Moreover, in this case we have $\cQ=H(\Cob(\bbone,\cC,\cC'))\circ\cP$.
	\end{itemize}
\end{crl}

\section{Application of the main result to decorated partition posets}

Let us apply this freeness theorem to decorated partition posets. We refer to \cite{OpPOS}  for the actual definition of the posets we consider.
Moreover, the following result is shown in \cite{OpPOS}:
\begin{thm}[\cite{OpPOS} Proposition 4.1]
	Let $\cP$ a set-operad, $\cC$ its dual cooperad, and $\Com^\vee$ the commutative cooperad, then:
	$$h(\Pi^\cP_+)=H(\Cob(\Com^\vee,\cC,\bbone))\qquad \text{ and } \qquad h( ^\cP\Pi^+)=H(\Cob(\bbone,\cC,\Com^\vee)).$$
\end{thm}
We summarise the poset we study with some of their properties in Table~\ref{tb:sum}. 

\begin{figure}[h]\caption{Summary table of posets considered in the article.}\label{tb:sum}
	\center
	\begin{tabular}{c|c|c|c|c}
		Poset & Associated & Koszul & Cohen-Macaulay & Combinatorial basis\\
		& set-operad & dual & & of the suspended\\
		& & & & homology\\
		\hline
		$\Pi^\Perm_+$ & $\Perm$ & $\PL$ & Yes & None\\[5pt]
		$ ^\DCom\Pi^+$ & $\DCom$ & $\DLie$ & Yes & None\\[5pt]
		$\Pi^\DCom_+$ & $\DCom$ & $\DLie$ & Yes & None\\[5pt]
		$\Pi^\ComTriAss_+$ & $\ComTriAss$ & $\PostLie$ & Yes & Binary trees\\[5pt]
		$ ^\NAC\Pi^+$ & $\NAC$ & $\NAC^!$ & No & unknown \\[5pt]
		$ ^\Dup\Pi^+$ & $\Dup$ & $\Dup^!$ & Conjecturally yes & None\\[5pt]
		$ \Pi^\TriAss_+$ & $\TriAss$ & $\TriDend$ & Conjecturally yes & None\\[5pt]
	\end{tabular}
\end{figure}

\subsection{The permutative operad}

Let us consider $\Pi^\Perm$. It is already shown in \cite{OpPOS}  that $\Pi^\Perm_+$ is Cohen-Macaulay. We recover this result since $\PL$ is free over $\Lie$ \cite{FPL}. 
Moreover, since it is shown in \cite{FPLr} that $\PL=\Lie\circ\cT(\CLie)$, we have the following result:

\begin{thm}\label{thm:pl}
	The right $\PL$-module $\Lambda\hat{h}(\Pi^\Perm)$ is isomorphic to $\cT(\CLie)$ with the module structure induced by the inclusion $\cT(\CLie)\hookrightarrow\PL$ given in \cite{FPL,FPLr}.
\end{thm}

\begin{rmk}
	In particular, we can check that the $\Sym_3$-representation $\cT(\CLie)(3)$ does not admit any basis stable by the action. Thus, there is no ``combinatorial basis'' describing $\Lambda\hat{h}(\Pi^\Perm)$.
\end{rmk}

\subsection{The double commutative operad}

It is shown in \cite{FreeOp}  that $\DLie$ is free both as a left and a right $\Lie$-module. 
Indeed, the Gröbner basis of $\DLie$ provided in \cite{FreeOp}[Theorem 2] can be used in \cite{FreeOp}[Theorem 4] for both left and right freeness. 
Thus, we have $\cX$, and $\cY$ such that $\DLie=\Lie\circ\cX$, and $\DLie=\cY\circ\Lie$.

\begin{thm}\label{thm:dlie}
	Both $ ^\DCom\Pi^+$ and $\Pi^\DCom_+$ are Cohen-Macaulay.
\end{thm}

\begin{rmk}
	Once again, one cannot hope for a ``combinatorial basis''. Indeed, one may check that both $\cX(2)$ and $\cY(2)$ is the signature representation of $\Sym_2$.
\end{rmk}

One may notice that we did not carefully state which map $\Lie\hookrightarrow\DLie$ we were using. 
This is irrelevant in this specific case, since all such map are equivalent in the sense that for any two such maps $f,g$, we have $\varphi$ an automorphism of $\DLie$ such that $f=\varphi(g)$.

\subsection{The commutative triassociative operad}

It was already pointed out in \cite{PostLie}  where $\PostLie$ is introduced that $\PostLie=\Lie\circ\Mag$ as species. One can easily check that this still hold when considering either of the two natural left $\Lie$-module structures of $\PostLie$.

\begin{thm}\label{thm:postl}
	The poset $\Pi^\ComTriAss_+$ is Cohen-Macaulay. Moreover, $\Lambda\hat{h}(\Pi^\ComTriAss)=\Mag$.
\end{thm}

\subsection{Other operads}

The operad $\NAC$ introduced in \cite{OpPOS} was also considered in \cite{BinSetOP}, and before in the PhD thesis of the author, where it is proven to defined by distributive law $\ANil\circ\CMag$. In particular, the morphism $\Lie\to\NAC$ is not injective, and $\NAC$ is not free as either a left or a right $\Lie$-module. Thus, although there is no obstruction when computing the Euler characteristic of $\check{h}( ^\NAC\Pi^+)$, it is not concentrated in the diagonal.

It is not know if $\Dup^!$ is left free over $\Lie$ nor if $\TriDend$ is right free over $\Lie$. 
It should be possible to answer those question using \cite{FreeOp}[Theorem 4] with a well-chosen Gröbner basis. 
Moreover, a positive answer would show the Cohen-Macaulayness of the corresponding poset, $ ^\Dup\Pi^+$ or $\Pi^\TriAss_+$, additionally, an explicit computation of the generators would compute the according homology groups.
One should notice that it is not possible to get ``combinatorial basis'' for those homology groups, since the dimension of the isotypic component of the signature representation is higher than the one of the trivial representation in the arity $2$ part of the suspended homology of those posets.

\subsection*{Funding}
{\small The author is funded by a postdoctoral fellowship of the ERC Starting Grant “Low Regularity Dynamics via Decorated Trees” (LoRDeT) of Yvain Bruned (grant agreement No. 101075208).}

\bibliographystyle{plain}
\bibliography{Bibly}
\textsc{Institut Élie Cartan de Lorraine, UMR 7502, Faculté des Sciences et Technologies, Boulevard des Aiguillettes, 54506 Vandœuvre-lès-Nancy, France}
\textit{Email address:} \texttt{paul.laubie@univ-lorraine.fr}

\end{document}